\documentclass[intlimits,12pt,reqno]{amsart}
\usepackage{latexsym,amsthm,amsmath,amssymb,mathrsfs,mathtools,upgreek}
\usepackage[T1]{fontenc}
\usepackage[utf8]{inputenc}
\usepackage[mathscr]{eucal}
\usepackage{enumerate}
\usepackage{enumitem}

\let\oldmarginpar\marginpar
\renewcommand{\marginpar}[2][rectangle,draw,text width= 2cm,rounded corners]{
    \oldmarginpar{
    \scriptsize \tikz \node at (0,0) [#1]{#2};}
    }
\usepackage{tikz}
\usetikzlibrary{arrows,calc,patterns,cd}
\usepackage{wrapfig}

\usepackage{hyperref}

\usepackage{cleveref}
\usepackage{xcolor}
\hypersetup{
    colorlinks=true,
    linkcolor=blue,
    citecolor=blue,
    urlcolor=blue
}

plus 6pt minus 12pt
\def\mvint_#1{\mathchoice
          {\mathop{\vrule width 6pt height 3 pt depth -2.5pt
                  \kern -9pt \intop}\limits_{\kern -3pt #1}}%
          {\mathop{\vrule width 5pt height 3 pt depth -2.6pt
                  \kern -6pt \intop}\nolimits_{#1}}%
          {\mathop{\vrule width 5pt height 3 pt depth -2.6pt
                  \kern -6pt \intop}\nolimits_{#1}}%
          {\mathop{\vrule width 5pt height 3 pt depth -2.6pt
                  \kern -6pt \intop}\nolimits_{#1}}}

\newcommand{\Ep}{\bigwedge\nolimits}

\newcommand{\bbbh}{\mathbb H}
\newcommand{\bbbr}{\mathbb R}
\newcommand{\bbbb}{\mathbb B}

\newcommand{\bbbz}{\mathbb Z}

\newcommand{\R}{\mathbb R}
\newcommand{\overbar}[1]{\mkern 1.7mu\overline{\mkern-1.7mu#1\mkern-1.5mu}\mkern 1.5mu}

\newcommand{\eps}{\varepsilon}

\newtheorem{theorem}{Theorem}[section]
\newtheorem*{theorem*}{Theorem}
\newtheorem{lemma}[theorem]{Lemma}
\newtheorem{corollary}[theorem]{Corollary}

\theoremstyle{definition}

\newtheorem{remark}[theorem]{Remark}
\newtheorem*{remark*}{Remark}

\newcommand{\Sph}{\mathbb S}

\newcommand{\supp}{{\rm supp}\,}

\usepackage{setspace}

\makeatletter
\renewcommand{\tocsection}[3]{%
  \indentlabel{\@ifnotempty{#2}{\bfseries\ignorespaces#1 #2\quad}}\bfseries#3}
\renewcommand{\tocsubsection}[3]{%
  \indentlabel{\@ifnotempty{#2}{\ignorespaces#1 #2\quad}}#3}

\newcommand\@dotsep{4.5}
\def\@tocline#1#2#3#4#5#6#7{\relax
  \ifnum #1>\c@tocdepth 
  \else
    \par \addpenalty\@secpenalty\addvspace{#2}%
    \begingroup \hyphenpenalty\@M
    \@ifempty{#4}{%
      \@tempdima\csname r@tocindent\number#1\endcsname\relax
    }{%
      \@tempdima#4\relax
    }%
    \parindent\z@ \leftskip#3\relax \advance\leftskip\@tempdima\relax
    \rightskip\@pnumwidth plus1em \parfillskip-\@pnumwidth
    #5\leavevmode\hskip-\@tempdima{#6}\nobreak
    \leaders\hbox{$\m@th\mkern \@dotsep mu\hbox{.}\mkern \@dotsep mu$}\hfill
    \nobreak
    \hbox to\@pnumwidth{\@tocpagenum{\ifnum#1=1\bfseries\fi#7}}\par
    \nobreak
    \endgroup
  \fi}
\AtBeginDocument{%
\expandafter\renewcommand\csname r@tocindent0\endcsname{0pt}
}
\def\l@subsection{\@tocline{2}{-5pt}{2.5pc}{5pc}{}}
\makeatother

\title{On the Gromov non-embedding theorem}

\author[Haj\l{}asz]{Piotr Haj\l{}asz}
\address{Piotr Haj\l{}asz,\newline \indent Department of Mathematics, University of Pittsburgh, \newline \indent 301 Thackeray Hall, Pittsburgh,
Pennsylvania 15260}
\email{hajlasz@pitt.edu}
\thanks{P.H. was supported by NSF grant  DMS-2055171 and 
by Simons Foundation grant
917582.}

\author[Schikorra]{Armin Schikorra}
\address{Armin Schikorra,\newline \indent Department of Mathematics, University of Pittsburgh, \newline \indent 301 Thackeray Hall, Pittsburgh,
Pennsylvania 15260}
\email{armin@pitt.edu}
\thanks{A.S. was supported by NSF Career DMS-2044898.}


\begin{document}

\begin{abstract}
We develop a new method leading to an elementary proof of a generalization of Gromov's theorem about non existence of H\"older embeddings into the Heisenberg group.
\end{abstract}

\maketitle

\section{Introduction}
\label{intro}
The Heisenberg group $\bbbh^n$ equipped with the Carnot-Carath\'eodory metric $d_c$ is homeomorphic to $\R^{2n+1}$, but its Hausdorff dimension equals $2n+2$. On compact sets the metric $d_c$ satisfies $|p-q|\lesssim d_c(p,q)\lesssim|p-q|^{1/2}$. In fact, in some directions the metric $d_c$ is comparable to the Euclidean metric, while in other directions it is of fractal nature. The Heisenberg group is also isomorphic to the standard contact structure on $\R^{2n+1}$, and the metric $d_c$ gives a natural way of measuring distances between points along curves tangent to the contact distribution. It is well known that every Legendrian (i.e., horizontal) map $f:U\subset\R^m\to\bbbh^n\cong\bbbr^{2n+1}$, that is smooth or  Lipschitz, must satisfy $\operatorname{rank} Df\leq n$, and hence there is no Lipschitz embedding of $U$ into $\bbbh^n$, if $m\geq n+1$. On the other hand, any smooth or Lipschitz map $f:U\to\bbbr^{2n+1}$, is $\frac{1}{2}$-H\"older continuous as a map into $\bbbh^n$. In this context Gromov in his seminal work \cite{gromov2}, initiated the program of investigating properties of H\"older continuous mappings into the Heisenberg group. He proved the following non-embedding result \cite[3.1.A]{gromov2}:
\begin{theorem}
\label{T1}
Every $\gamma$-H\"older continuous embedding $f:\R^m\to\bbbh^n$, $m\geq n+1$, satisfies $\gamma\leq\frac{n+1}{n+2}$.
\end{theorem}
The original proof is actually quite difficult since it requires the $h$-principle and microflexibility. 
For a more detailed presentation of this proof and another proof based on the Rumin complex, see \cite{pansu}. There is also a proof of the non-existence of an embedding $f\in C^{0,\gamma}(\R^2,\bbbh^1)$, $\gamma>2/3$, due to Z\"ust \cite{zust1}, based on a factorization through a topological tree. However, the argument cannot be adapted to the general case of Theorem~\ref{T1}.

Gromov \cite{gromov2}, conjectured that if $f:\R^2\to\bbbh^1$ is a $\gamma$-H\"older embedding, then $\gamma\leq \frac{1}{2}$. Back in 2015, Haj\l{}asz, Mirra, and Schikorra ran numerical experiments suggesting that in fact an embedding could possibly exist for all $\gamma<\frac{2}{3}$. They provided numerical support indicating that any smooth map $f:\Sph^1\to\bbbh^1$ admits a $\gamma$-H\"older extension $f:\bbbb^2\to\bbbh^1$ if $\gamma\in (0,\frac{2}{3})$. Recently, Wenger and Young \cite{wengery}, proved that the pair $(\bbbr^2,\bbbh^1)$ has the $\gamma$-H\"older extension property for any $\gamma\in (0,\frac{2}{3})$. In particular, any mapping $f\in C^{0,\gamma}(\Sph^1,\bbbh^1)$ admits a $C^{0,\gamma}$-extension to $\bbbb^2$. However, their construction will not lead to an embedding.

In the last two decades, there has been a substantial interest in studying H\"older continuous functions and mappings from the analytic and geometric point of view.
While H\"older continuous mappings need not be differentiable at any point, a new approach to analysis of such maps has been developed by Brezis and Nguyen \cite{brezisn1,brezisn2}, as an abstract theory, Conti, De Lellis and Sz\'ekelyhidi \cite{contids}, with applications to rigidity of isometric embeddings, and Z\"ust \cite{zust2}, in the context of currents. 
The theory plays also an important role in Isett's solution of the Onsager conjecture \cite{isett}.
See also \cite{sickel2,sickel1,LS20}, for related research.
For a detailed study of Jacobians and pullbacks of differential forms by H\"older continuous mappings with applications to the geometry and topology of the Heisenberg groups, see \cite{HMS}.
This paper develops at depth the theory related to the proof of Theorem~\ref{T2} presented below.

The purpose of this work is to provide an elementary proof of the following generalization of Gromov's non-embedding theorem.
\begin{theorem}
\label{T2}
Suppose that $m\geq n+1$,
$$
\frac{1}{2}\leq\gamma\leq\frac{n+1}{n+2}
\quad
\text{and}
\quad
\theta=\frac{n+1}{n}-\frac{2\gamma}{n}.
$$
Then, there does not exist a map
$$
f\in C^{0,\gamma+}({\bbbb}^m,\bbbh^n)\cap C^{0,\theta}({\bbbb}^m,\R^{2n+1})
$$
or a map
$$
f\in C^{0,\gamma}({\bbbb}^m,\bbbh^n)\cap C^{0,\theta+}({\bbbb}^m,\R^{2n+1}),
$$
such that $f|_{\Sph^{m-1}}$ is a topological embedding. 
\end{theorem}
Here, $C^{0,\alpha}({\bbbb}^m,X)$ stands for the class of H\"older continuous mappings into a metric space $X$.
Moreover, $f\in C^{0,\alpha+}({\bbbb}^m,X)$ means that 
$$
\lim_{t\to 0^+}\ \sup\left\{\frac{d(f(x),f(y))}{|x-y|^\alpha}:\ x,y\in {\bbbb}^m,\ 0<|x-y|\leq t\right\}= 0.
$$
This is a natural class of mappings: uniformly continuous mappings are precisely mappings of the class $C^{0,0+}$ and if $\alpha\in (0,1)$, then $C^{0,\alpha+}(\R^m)$ is the subspace of $C^{0,\alpha}(\R^m)$
that consists of functions that can be approximated by $C^\infty$ functions in the $C^{0,\alpha}$ norm on compact sets.

If $\gamma=\frac{n+1}{n+2}$, then $\theta=\frac{n+1}{n+2}$, and we obtain the following strengthened version of Gromov's theorem:
\begin{corollary}
\label{T3}
If $m\geq n+1$, then there is no map $f\in C^{0,\frac{n+1}{n+2}+}(\bbbb^m,\bbbh^n)$ such that $f|_{\Sph^{m-1}}$ is a topological embedding. 
\end{corollary}
The recent work of Wenger and Young \cite{wengery}, shows that \Cref{T3} is sharp when $n=1$: for any $\gamma < \frac{2}{3}$, any $C^{0,\gamma}$-embedding of the the $\mathbb{S}^1$-sphere into $\mathbb{H}^1$ can be extended to a map $f \in C^{0,\gamma}({\bbbb^2},\bbbh^1)$.

When $\gamma=\frac{1}{2}$, then $\theta=1$, and we obtain in particular the following result from \cite[Theorem~1.11]{BHW} that was proved by different techniques:
\begin{corollary}
\label{T4}
If $m\geq n+1$, then there is no embedding $f\in C^{0,\frac{1}{2}+}({\bbbb^m},\bbbh^n)\cap C^{0,1}({\bbbb^m},\R^{2n+1})$.
\end{corollary}

\begin{remark}
In this work we focus our attention on H\"older continuous maps both for simplicity and because Gromov's theorem is stated for H\"older maps.
However, since the proof implicitly relies on distributional Jacobians and distributional pullbacks of forms by H\"older mappings, the results obtained here can likely be generalized to limiting fractional Sobolev spaces. 
For this one should combine the present arguments with the techniques developed in \cite{sickel1,brezisn1,LS20,SVS20}. Also, there is likely a limiting Besov-space result containing all previous examples similar to the limiting Besov space in the Onsager conjecture, \cite{CET94}.
\end{remark}

\subsection*{Acknowledgements} We would like to express our gratitude to Jacob Mirra for his valuable insights and helpful discussions, which greatly assisted us in the development of this paper.

\section{Preliminaries}
In this section we recall basic definitions and we collect auxiliary results needed in the proof of Theorem~\ref{T2}. 

For non-negative expressions $A$ and $B$ we write $A\lesssim B$ if $A\leq CB$ for some constant $C>0$ and $A\approx B$ if $A\lesssim B\lesssim A$. The unit ball and the balls of radius $r$ centered at $0$ will be denoted by $\bbbb^m$ and $\bbbb^m_r$. $\Sph^m$ will stand for the unit sphere centered at the origin. $U$ will always be an open subset in $\R^m$. We will denote the $L^p$ norm by $\Vert f\Vert_p$ or $\Vert f\Vert_{L^p(U)}$.

If $f\in C^{0,\gamma}(U)$, then we write
$$
[f]_\gamma=[f]_{\gamma,U}=\sup_{\substack{x,y\in U\\x\neq y}}\frac{|f(x)-f(y)|}{|x-y|^\gamma}\, ,
\quad
\text{and}
\quad
[f]_{\gamma,\eps}=[f]_{\gamma,\eps,U}=\sup_{\substack{x,y\in U\\0<|x-y|\leq\eps}}
\frac{|f(x)-f(y)|}{|x-y|^\gamma}\, .
$$
Note that if $f\in C^{0,\gamma+}_{\rm loc}(U)$, then for any $U'\Subset U$ we have that $[f]_{\gamma,\eps,U'}\to 0$ as $\eps\to 0$.

\subsection{Pullbacks of differential forms}
The spaces of smooth differential $k$-forms on $U$, and smooth $k$-forms with compact support in $U$, will be denoted by $\Omega^k(U)$ and $\Omega_c^k(U)$ respectively. If the coefficients of the differential forms will be in $L^\infty$, $C^{0,\gamma}$ etc.\ we will write 
$\Omega^k L^\infty(U)$, $\Omega^kC^{0,\gamma}(U)$ etc.

Let $\phi\in C_c^\infty(\bbbb^m)$ be a standard mollifier and $\phi_\eps(x)=\eps^{-m}\phi(x/\eps)$. Then
$f_\eps:=f*\phi_\eps$ satisfies
$$
Df_\eps(x)=\eps^{-m-1}\int_{\bbbb^m_\eps}(f(x-z)-f(x))D\phi\Big(\frac{z}{\eps}\Big)\, dz,
\quad
\Vert Df_\eps\Vert_{L^\infty(\bbbb^m(x_o,r))}\lesssim\eps^{\gamma-1}[f]_{\gamma,\eps,\bbbb^m(x_o,2r)},
$$
provided $0<\eps<r$. This estimate easily implies
\begin{lemma}
\label{T5}
If $f\in C^{0,\gamma}(U,\R^d)$, $U\subset\R^m$, $\gamma\in (0,1]$, and $B(x_o,2r)\subset U$, then for any
$\kappa\in\Omega^k L^\infty(\R^d)$ we have
$$
\Vert f_\eps^*\kappa\Vert_{L^\infty(\bbbb^m(x_o,r))}\lesssim \Vert\kappa\Vert_\infty [f]^k_{\gamma,\eps,\bbbb^m(x_o,2r)}
\eps^{k(\gamma-1)},
\quad
\text{provided } 0<\eps<r.
$$
The implied constant depends on $m$, $d$, $k$ and $\phi$ only.
\end{lemma}

\subsection{The Heisenberg group}
We assume that the reader is familiar with the Heisenberg groups, but we recall some definitions to fix notation. 
Denote coordinates in $\bbbr^{2n+1}$ by $(x,y,t)\in\bbbr^{n}\times\bbbr^{n}\times\bbbr$.
The Heisenberg group $\bbbh^n=\bbbr^{2n+1}$ is a Lie group with the group law
$$
(x,y,t)*(x',y',t')
=
\Big(x+x',y+y',t+t'+2\sum_{j=1}^n(y_jx_j'-x_jy_j')\Big).
$$
The Heisenberg group is equipped with the $2n$-dimensional horizontal distribution which is the kernel of the standard contact form 
\[
\upalpha=dt+2\sum_{j=1}^n (x_jdy_j-y_jdx_j).
\] The Carnot-Carath\'eodory metric is defined as infimum of lengths of horizontal curves (i.e. curves tangent to the horizontal distribution) that connect given two points. The length of such curves is computed with respect to a left invariant Riemannian metric. In other words the Heisenberg group is the standard contact structure on $\R^{2n+1}$ equipped with a natural metric of measuring distances along Legendrian curves. 

It turns out that the metric $d_c$ is bi-Lipschitz equivalent with the Kor\'{a}nyi metric
$$
d_K(p,q) := \Vert q^{-1} * p \Vert_K,
\quad
\text{where}
\quad
\Vert (x,y,t) \Vert_K := (|(x,y)|^4+t^2)^{1/4}.
$$
In what follows we will assume that $\bbbh^n$ is a metric space with respect to the metric $d_K$. It will be convenient to write
\begin{equation}
\label{eq1}
d_K(p,q)=(|\pi(p)-\pi(q)|^4+|\varphi(p,q)|^2)^{1/4}
\approx
|\pi(p)-\pi(q)|+|\varphi(p,q)|^{1/2},
\end{equation}
where $\pi:\R^{2n+1}\to\R^{2n}$, $\pi(x,y,t)=(x,y)$ is the orthogonal projection and
\begin{equation}
\label{eq2}
\varphi(p,q)=t-t'+2\sum_{j=1}^n (x_j'y_j - x_jy_j') \quad \text{for $p=(x,y,t)$, $q= (x',y',t')$}.
\end{equation}
It is important to note that the expression $\varphi(p,q)$ is somewhat similar to the contact form $\upalpha$. It will play an important role in our proofs.

Recall that if $K\subset\R^{2n+1}$ is a compact set, then there is a constant $C\geq 1$ such that
\begin{equation}
\label{eq3}
C^{-1}|p-q|\leq d_K(p,q)\leq C|p-q|^{1/2}
\quad
\text{for } p,q\in K.
\end{equation}
In particular, the identity map $\operatorname{id}:\bbbh^n\to\R^{2n+1}$ is locally Lipschitz continuous.
\begin{lemma}
\label{T7}
Suppose that
$f\in C^{0,\gamma}(U,\bbbh^n)$,  $U\subset\R^m$, and
$\gamma\in (0,1]$. If $\bbbb^m(x_o,2r)\subset U$, then
$$
\Vert f^*_\eps\upalpha\Vert_{L^\infty(\bbbb^m(x_o,r))}\leq 
4^\gamma\Vert D\phi\Vert_1[f]_{\gamma,2\eps,\bbbb^m(x_o,2r)}^2\,\eps^{2\gamma-1}
\quad
\text{for all $0<\eps<r$.}
$$
\end{lemma}
\begin{proof}
Using notation $f=(f^x,f^y,f^t)$, \eqref{eq2}, and the fact that for a fixed point $p\in \bbbb^m(x_o,r)$, we have
$\int_{\bbbb^m_\eps}f_\eps^t(p)d\phi(z/\eps)\, dz=0$, the following equalities are easy to verify
\begin{equation*}
\begin{split}
&(f^*_\eps\upalpha)(p)=df^t_\eps(p)+2\sum_{j=1}^n (f_\eps^{x_j}(p)d f_\eps^{y_j}(p)-f_\eps^{y_j}(p)d f_\eps^{x_j}(p))=\\
&
\eps^{-m-1}\int_{\bbbb^m_\eps}
\left( f^t(p-z)-f_\eps^t(p)+
2\sum_{j=1}^n\big((f_\eps^{x_j}(p)f^{y_j}(p-z)-f_\eps^{y_j}(p)f^{x_j}(p-z)\big)\right)d\phi
\left(\frac{z}{\eps}\right)\, dz\\
&=
\eps^{-m-1}\int_{\bbbb^m_\eps}\int_{\bbbb^m_\eps} \varphi(f(p-z),f(p-w))\phi_\eps(w)d\phi\left(\frac{z}{\eps}\right)\, dw\, dz.
\end{split}    
\end{equation*}
Since
$$
|\varphi(f(p-z),f(p-w))|\leq d_K(f(p-z),f(p-w))^2\leq [f]_{\gamma,2\eps,\bbbb^m(x_o,2r)}^2(2\eps)^{2\gamma},
$$
the result easily follows.
\end{proof}
\begin{remark}
One can show that \Cref{T7} implies that for $f \in C^{0,\frac{1}{2}+}(U,\bbbh^n)$, we have 
$
 f^\ast(\upalpha) = 0
$
in the distributional sense. If $f \in C^{0,1}(U,\bbbh^n)$, this is also true in a.e. sense. For more details, see \cite{HMS}.
\end{remark}

\subsection{Decomposition of forms}
In what follows $\upomega=\sum_{j=1}^n dx_j\wedge dy_j\in\Ep^2(\R^{2n})^*$ will denote the standard symplectic form. Then next algebraic result due to Lefschetz is well known, see \cite[Proposition~1.1a]{Bryant} or \cite[Proposition~1.2.30]{huybrechts}.
\begin{lemma}
\label{T8}
The wedge product with the symplectic form
\begin{equation}
\label{eq4}
\upomega\wedge\cdot:\Ep^{n-1}(\R^{2n})^*\to \Ep^{n+1}(\R^{2n})^*,
\end{equation}
is an isomorphism.
\end{lemma}
As a straightforward application of Lemma~\ref{T8} we obtain
\begin{lemma}
\label{T9}
If $\kappa\in \Omega^{n+1}(\R^{2n+1})$, then there are smooth forms
$\beta\in \Omega^{n}(\R^{2n+1})$ and 
$\delta\in \Omega^{n-1}(\R^{2n+1})$
that do not contain components with $dt$ and satisfy
$$
\kappa=\upalpha\wedge\beta+d\upalpha\wedge\delta.
$$
Moreover, if $\kappa(p)=0$, then $\beta(p)=0$ and $\delta(p)=0$, i.e., $\supp\beta\cup\supp\delta=\supp\kappa$.
\end{lemma}
\begin{proof}
If $\pi:\R^{2n+1}\to\R^{2n}$, $\pi(x,y,t)=(x,y)$, then $\pi^*\upomega\in\Omega^2(\R^{2n+1})$ is a natural identification of the symplectic form with a $2$-form on $\R^{2n+1}$. Note that
$d\upalpha
=4\sum_{j=1}^n dx_j\wedge dy_j=4\pi^*\upomega$.

With the notation $dx_I:=dx_{i_1}\wedge\ldots\wedge dx_{i_k}$, $I=(i_1,\ldots,i_k)$, $1\leq i_1<\ldots<i_k\leq n$, $|I|=k$, it easily follows from Lemma~\ref{T8} that if $\eta\in\Omega^{n+1}(\R^{2n+1})$, and
$
\eta=\sum_{|I|+|J|=n+1} \eta_{IJ}\, dx_I\wedge dy_J,
$
i.e., if $\eta$ is a form that does not contain components with $dt$, then $\eta$ can be uniquely written as $\eta=d\upalpha\wedge\delta$, where $\delta$ does not contain components with $dt$. Since the coefficients of $\delta$ are obtained from the coefficients of $\eta$ through the isomorphism \eqref{eq4} of linear spaces, it follows that $\delta\in\Omega^{n-1}(\R^{2n+1})$ and $\delta(p)=0$ whenever $\eta(p)=0$.

Since $\{\upalpha(p),dx_1,dy_1,\ldots,dx_n,dy_n\}$ is a basis of $T^*_p\R^{2n+1}$, it follows that 
elements of the form
$$
dx_I\wedge dy_J, \
|I|+|J|=n+1, 
\quad
\text{and}
\quad
\upalpha(p)\wedge dx_{I'}\wedge dy_{J'},
\
|I'|+|J'|=n,
$$
form a basis of $\Ep^{n+1}T_p^*\R^{2n+1}$. Therefore, any form $\kappa\in\Omega^{n+1}(\R^{2n+1})$ can be uniquely represented as $\eta+\upalpha\wedge\beta$, where at every point, $\eta$ belongs to the span of $dx_I\wedge dy_J$, $|I|+|J|=n+1$ and $\beta$ belongs to span of $dx_{I'}\wedge dy_{J'}$, $|I'|+|J'|=n$. Hence,
$$
\kappa=\eta+\upalpha\wedge\beta=d\upalpha\wedge\delta +\upalpha\wedge\beta.
$$
Since the coefficients of the forms $\eta$ and $\beta$, and hence coefficients of the forms $\delta$ and $\beta$ are obtained from coefficients of $\kappa$ (in the standard basis) by linear isomorphisms, it follows that the forms $\delta\in\Omega^{n-1}(\R^{2n+1})$, $\beta\in\Omega^n(\R^{2n+1})$ are smooth, and $\delta(p)=0$, $\beta(p)=0$, whenever $\kappa(p)=0$.
\end{proof}

\subsection{Poincar\'e and Alexander dualities} 
Let $f:\Sph^k\to\R^m$, $k<m$, be a topological embedding. It follows from the Poincar\'e and the Alexander dualities that $H^{k+1}_c(\R^m\setminus{}f(\Sph^k))=\bbbz$. Thus, translating it into the de~Rham cohomology, there is a closed form $\kappa\in\Omega_c^{k+1}(\R^m\setminus f(\Sph^k))$ that cannot be written as $\kappa=d\eta$, for any $\eta\in\Omega^k_c(\R^m\setminus f(\Sph^k))$. The existence of $\kappa$ can be proved using the Mayer-Vietoris sequence. In the next result we construct $\kappa$ directly and explicitly. Our proof is elementary and self-contained, but it mimics the construction of the Mayer-Vietoris sequence.
\begin{lemma}
\label{T10}
Let $f:\Sph^k\to\R^m$, $k<m$, be a topological embedding. Then, there is a closed form $\kappa\in\Omega_c^{k+1}(\R^m\setminus f(\Sph^k))$ such that if $f_i\in C^\infty(\R^{k+1},\R^m)$ 
and $f_i|_{\Sph^k}$ converges uniformly to $f$, then there is $i_o$ such that
\begin{equation}
\label{eq5}
 \int_{\bbbb^{k+1}} f_i^*\kappa =1
\quad
\text{for all } i\geq i_o.
\end{equation}
\end{lemma}
\begin{remark}
Note that $\kappa\neq d\eta$ for all
$\eta\in\Omega_c^k(\R^m\setminus f(\Sph^k))$, as otherwise integral \eqref{eq5} would be zero for large $i$, by Stokes' theorem.
\end{remark}
\begin{proof}
Fix $m$ throughout the proof and proceed by induction on $k$.
For $k=0$ we have $\Sph^0=\{-1,+1\}$ and we take $\eta\in C_c^\infty(\R^m)$ that equals $1$ in a neighborhood of $f(+1)$ and equal $0$ in a neighborhood of $f(-1)$. Then, we define $\kappa=d\eta\in \Omega_c^1(\R^m\setminus f(\Sph^0))$. Clearly, $\kappa$ is closed. If $i$ is sufficiently large, then
$$
\int_{\bbbb^1} f_i^*\kappa=\int_{-1}^1f_i^*d\eta=
\eta(f_i(+1))-\eta(f_i(-1))=1.
$$
Suppose now that the result is true for $k-1$ and we will prove it for $k$.
Let $\Sph^k_+$, $\Sph^k_-$ and $\Sph^{k-1}=\Sph^k_+\cap\Sph^k_-$, be the closed hemispheres and the equator of $\Sph^k$, respectively. Let $\bbbb^k\subset\bbbb^{k+1}$, $\partial\bbbb^k=\Sph^{k-1}$, be the intersection of $\bbbb^{k+1}$ with the equatorial hyperplane. 

By the induction hypothesis we can find a closed form
$\tilde{\kappa}\in \Omega_c^k(\R^m\setminus f(\Sph^{k-1}))$ and $i_o$, such that 
\begin{equation}
\label{eq6}
\int_{\Sph^k_+} f_i^*\tilde{\kappa}=\int_{\bbbb^k}f_i^*\tilde{\kappa}=1,
\qquad
\text{for all } i\geq i_o,
\end{equation}
where the first equality follows from the fact that $\Sph^k_+\cup\bbbb^k$ is the boundary of a half-ball and the integral of the closed form $f_i^*\tilde{\kappa}$ on that boundary of the half-ball equals zero. Of course, we need to choose a proper orientation of $\Sph^k_+$.

Let $U$ be an open neighborhood of $f(\Sph^{k-1})$ such that $\supp\tilde{\kappa}\subset \R^m\setminus\overbar{U}$. Since
$$
\{\R^m\setminus f({\Sph}^k_-),\ \R^m\setminus f({\Sph}^k_+),\ U\}
$$ 
is an open covering of $\R^m$, we may choose a smooth partition of unity $\{\psi_+,\psi_-,\psi_U\}$ subordinate to that covering.
Let $\sigma=\psi_+\tilde{\kappa}$, and $\delta=\psi_-\tilde{\kappa}$. Since $\psi_U\tilde{\kappa}=0$, we have that $\tilde{\kappa}=\sigma+\delta$. 
Since $d\tilde{\kappa}=0$, it follows that $d\sigma=-d\delta$, so $\supp d\sigma=\supp d\delta$.
Note that $\sigma=0$ is a neighborhood of $f({\Sph}^k_-)$ and $\delta =0$ in a neighborhood of $f({\Sph}^k_+)$, so $d\sigma=-d\delta=0$ in a neighborhood of $f(\Sph^k)$ i.e.,
$d\sigma\in\Omega_c^{k+1}(\R^m\setminus f(\Sph^k))$. Moreover, the observation about supports of the forms $\sigma$ and $\delta$ along with \eqref{eq6} yields
$$
\int_{\bbbb^{k+1}} f_i^*d\sigma=\int_{\Sph^k} f_i^*\sigma=
\int_{\Sph^k_+} f_i^*\sigma=
\int_{\Sph^k_+} f_i^*(\sigma+\delta)=
\int_{\Sph^k_+} f_i^*\tilde{\kappa}=1
\quad
\text{for all } i\geq i_o.
$$
Therefore, the closed form
$\kappa:=d\sigma\in\Omega_c^{k+1}(\R^m\setminus f(\Sph^k))$ satisfies the claim.
The proof is complete.
\end{proof}

\section{Proof of Theorem~\ref{T2}}
Since $\overbar{\bbbb}^{n+1}\subset \overbar{\bbbb}^m$, an embedding of $\Sph^{m-1}\subset\overbar{\bbbb}^m$ yields an embedding of $\Sph^{n}\subset\overbar{\bbbb}^{n+1}$. Therefore, without loss of generality, we may assume that $m=n+1$. Suppose, for the sake of contradiction, that there is a map
$$
f\in C^{0,\gamma+}({\bbbb}^{n+1},\bbbh^n)\cap
C^{0,\theta}({\bbbb}^{n+1},\R^{2n+1})
\ \
\text{or}
\ \
f\in C^{0,\gamma}({\bbbb}^{n+1},\bbbh^n)\cap
C^{0,\theta+}({\bbbb}^{n+1},\R^{2n+1}),
$$
such that $f|_{\Sph^n}$ is an embedding. In order to approximate $f$ by convolution all the way up to the boundary, it would be convenient to assume that $f$ is defined on $\R^{n+1}$, but we can easily do that, because the extension of $f$ by $f(x)=f(x/|x|)$, when $|x|>1$, belongs to the same H\"older classes as the original map $f$. Let $f_i=f*\phi_{1/i}$. Clearly, $f_i$ restricted to $\Sph^n$ converges uniformly to $f|_{\Sph^n}$.

Let $\kappa\in \Omega_c^{n+1}(\R^{2n+1}\setminus f(\Sph^n))$ be as in Lemma~\ref{T10} (with $m=2n+1$ and $k=n$), so 
$$
\int_{\bbbb^{n+1}} f_i^*\kappa=1
\quad
\text{for all } i\geq i_o.
$$
According to Lemma~\ref{T9}, we have
$\kappa=\upalpha\wedge\beta +d\upalpha\wedge\delta$ for some $\beta\in\Omega_c^n(\R^{2n+1}\setminus f(\Sph^n))$ and $\delta\in\Omega_c^{n-1}(\R^{2n+1}\setminus f(\Sph^n))$. 
By increasing $i_o$ if necessary, we may assume that $f_i(\Sph^n)\cap\supp\delta=\varnothing$, and hence, $f_i^*\delta|_{\bbbb^{n+1}}\in \Omega_c^{n-1}(\bbbb ^{n+1})$ for $i\geq i_o$, so integration by parts yields
$$
\int_{\bbbb^{n+1}}df_i^*\upalpha\wedge f_i^*\delta=
\int_{\bbbb^{n+1}} f_i^*\upalpha\wedge f_i^*(d\delta). 
$$
Therefore,
\begin{align*}
1
&=
\int_{\bbbb^{n+1}} f_i^*\kappa=
\int_{\bbbb^{n+1}}
\big(f^*_i\upalpha\wedge f_i^*\beta+df_i^*\upalpha\wedge f_i^*\delta\big)
=
\int_{\bbbb^{n+1}} f_i^*\upalpha\wedge f_i^*(\beta+d\delta) \\
&\lesssim 
\Vert f_i^*\upalpha\Vert_{L^\infty(\bbbb^{n+1})} \Vert f_i^*(\beta+d\delta)\Vert_{L^\infty(\bbbb^{n+1})}.
\end{align*}
This however, yields a contradiction, because
Lemmata~\ref{T5} and~\ref{T7} yield the estimates
\[
\begin{split}
\Vert f_i^*\upalpha\Vert_{\infty}&
\Vert f_i^*(\beta+d\delta)\Vert_{\infty}
\lesssim
i^{1-2\gamma}[f]^2_{\gamma,2i^{-1},\bbbb^{n+1}_2}\cdot i^{n(1-\theta)} [f]^n_{\theta,i^{-1},\bbbb^{n+1}_2} \\
&=
[f]^2_{\gamma,2i^{-1},\bbbb^{n+1}_2}[f]^n_{\theta,i^{-1},\bbbb^{n+1}_2}\to 0
\quad
\text{as $i\to\infty$,}
\end{split}
\]
because $f\in C^{0,\gamma}\cap C^{0,\theta}$ and $f\in C^{0,\gamma+}$ or $f\in C^{0,\theta+}$. 

\qed


\end{document}